\documentclass[a4paper,12pt]{article}
\usepackage{amsmath}
\usepackage{amssymb}
\usepackage{amscd}
\usepackage{amsthm}
\usepackage{graphicx}
\usepackage{setspace}
\usepackage{bm}
\usepackage{amsthm}
\usepackage{pdfsync}
\newtheorem{theorem}{Theorem}[section]
\newtheorem{lemma}[theorem]{Lemma}
\newtheorem*{rtheorem}{Main Theorem}
\newtheorem{corollary}[theorem]{Corollary}

\theoremstyle{definition}

\newtheorem*{definition}{Definition}

\newtheorem*{remark}{Remark}

\def\PG{\mathrm{PG}} \def\AG{\mathrm{AG}}

\def\Aut{\mathrm{Aut}}
\def\PGammaL{\mathrm{P}\Gamma\mathrm{L}}
\def\PGL{\mathrm{PGL}}

\def\A{\mathcal{A}} \def\B{\mathcal{B}} \def\C{\mathcal{C}}
\def\D{\mathcal{D}}  
 \def\K{\mathcal{K}}
\def\L{\mathcal{L}}  
\def\Ll{\mathbb{L}}
\def\O{\mathcal{O}} \def\P{\mathcal{P}}
\def\H{\mathcal{H}}
\def\T{\mathcal{T}}
\def\R{\mathcal{R}} \def\S{\mathcal{S}}

\def\F{\mathbb{F}}

\title{Pseudo-ovals in even characteristic \\and ovoidal Laguerre planes}
\author{Sara Rottey \and Geertrui Van de Voorde \thanks{This author is a Postdoctoral Fellow of the Research Foundation -- Flanders (FWO -- Vlaanderen)}}
\begin{document}
\maketitle
\begin{abstract} Pseudo-arcs are the higher dimensional analogues of arcs in a projective plane: a {\em pseudo-arc} is a set $\A$ of $(n-1)$-spaces in $\PG(3n-1,q)$ such that any three span the whole space. Pseudo-arcs of size $q^n+1$ are called {\em pseudo-ovals}, while pseudo-arcs of size $q^n+2$ are called {\em pseudo-hyperovals}. A pseudo-arc is called {\em elementary} if it arises from applying field reduction to an arc in $\PG(2,q^n)$.

We explain the connection between dual pseudo-ovals and {\em elation Laguerre planes} and show that an elation Laguerre plane is {\em ovoidal} if and only if it arises from an elementary dual pseudo-oval. The main theorem of this paper shows that a pseudo-(hyper)oval in $\PG(3n-1,q)$, where $q$ is even and $n$ is prime, such that every element induces a Desarguesian spread, is elementary. As a corollary, we give a characterisation of certain ovoidal Laguerre planes in terms of the derived affine planes.
%
%We also state a conjecture on hyperovals in $\PG(2,q^n)$ and show that our main theorem for arbitrary $n$ follows if this conjecture holds true.
\end{abstract}

\section{Introduction}
The aim of this paper is to characterise elementary pseudo-(hyper)ovals in $\PG(3n-1,q)$ where $q$ is even. We will impose a condition on the considered pseudo-ovals, namely that every element of the pseudo-oval induces a Desarguesian spread. In Subsection \ref{S2}, we provide the necessary background on pseudo-arcs and give some motivation for the study of this problem. In Subsection \ref{S1}, we will introduce Desarguesian spreads and field reduction and prove a theorem on the possible intersection of Desarguesian $(n-1)$-spreads in $\PG(2n-1,q)$. In Section \ref{S3}, we will explain the connection between dual pseudo-ovals and elation Laguerre planes, meanwhile proving a theorem that characterises ovoidal Laguerre planes as those elation Laguerre planes obtained from an elementary dual pseudo-oval. Finally, in Section \ref{S4}, we give a proof for our main theorem. 
%We will prove this theorem in a setting for general $n$; we will formulate a conjecture on hyperovals in $\PG(2,q^n)$ which holds for $n$ prime and which would have the statement for general $n$ as a corollary. 
We end by stating a corollary of our main theorem in terms of ovoidal Laguerre planes.

\subsection{Pseudo-arcs}\label{S2}
In this paper, all considered objects will be finite. Denote the $n$-dimensional projective space over the finite field $\F_q$ with $q$ elements, $q=p^h$, $p$ prime, by $\PG(n,q)$.

\begin{definition} A {\em pseudo-arc} is a set $\A$ of $(n-1)$-spaces in $\PG(3n-1,q)$ such that $\langle E_i,E_j\rangle \cap E_k=\emptyset$ for distinct $E_i,E_j,E_k$ in $\A$.
\end{definition}
We see that a pseudo-arc is a set of $(n-1)$-spaces such that any $3$ span $\PG(3n-1,q)$; such a set is also called a set of $(n-1)$-spaces in $\PG(3n-1,q)$ {\em in general position}. 

A {\em partial spread} in $\PG(2n-1,q)$ is a set of mutually disjoint $(n-1)$-spaces in $\PG(2n-1,q)$. Every element $E_i$ of a pseudo-arc $\A$ defines a partial spread $$\S_{i}:=\{ E_1,\ldots,E_{i-1},E_{i+1},\ldots,E_{\vert \A\vert}\}/E_i$$ in $\PG(2n-1,q)\cong \PG(3n-1,q)/E_i$ and we say that the element $E_i$ {\em induces} the partial spread $\S_i$.  Since an element $E_i$ induces a partial spread $\S_i$ in $\PG(2n-1,q)$, which has at most $q^n+1$ elements, a pseudo-arc in $\PG(3n-1,q)$ can have at most $q^n+2$ elements. Moreover, we have the following theorem of Thas, where a {\em pseudo-oval} in $\PG(3n-1,q)$ denotes a pseudo-arc of size $q^n+1$, and a {\em pseudo-hyperoval} denotes a pseudo-arc of size $q^n+2$. Note that for $n=1$, these statements reduce to well-known and easy to prove statements.
\begin{theorem}{\rm \cite{Thas PHO}}\label{Thas PHO} A pseudo-arc in $\PG(3n-1,q)$, $q$ odd, has at most $q^n+1$ elements. A pseudo-oval in $\PG(3n-1,q)$, $q$ even, is contained in a unique pseudo-hyperoval.
\end{theorem}

A pseudo-arc is called {\em elementary} if it arises by applying field reduction to an arc in $\PG(2,q^n)$. {\em Field reduction} is the concept where a point in $\PG(2,q^n)$ corresponds in a natural way to an $(n-1)$-space of $\PG(3n-1,q)$. The set of all points of $\PG(2,q^n)$ then correspond to a set of disjoint $(n-1)$-spaces partitioning $\PG(3n-1,q)$, forming a {\em Desarguesian spread}. For more information on field reduction and Desarguesian spreads we refer to \cite{FQ11}. A pseudo-oval that is obtained by applying field reduction to a conic in $\PG(2,q^n)$ is called a {\em pseudo-conic}. A pseudo-hyperoval (necessarily in even characteristic) obtained by applying field reduction to a conic, together with its nucleus, is called a {\em pseudo-hyperconic}.

All known pseudo-ovals and pseudo-hyperovals are elementary, but it is an open question whether there can exist non-elementary pseudo-ovals and pseudo-hyperovals. A natural question to ask is whether we can characterise a pseudo-oval in terms of the partial spreads induced by its elements.

%In \cite{Penttila} the following theorem is proven:
%\begin{result}
%If $\K=\{K_1,\ldots,K_s\}$ is a pseudo-arc in $\PG(3n-1,q)$, $q$ odd, of size at least the size of the second largest complete arc in $\PG(2,q^n)$,  where for one element $K_i$ of $\K$, the partial spread $\S=\{K_1,\ldots,K_{i-1},K_{i+1},\ldots,K_{s}\}/K_i$ extends to a Desarguesian spread of $\PG(2n-1,q)=\PG(3n-1,q)/K_i$, then $\K$ is contained in a pseudo-conic.
%\end{result}

From \cite{Beutelspacher}, we know that a partial spread of $\PG(2n-1,q)$ of size $q^n$ can be extended to a spread in a unique way, i.e. the set of points in $\PG(2n-1,q)$ not contained in an element of such a partial spread of size $q^n$, form an $(n-1)$-space. So by abuse of notation, we say that an element of a pseudo-oval induces a spread instead of a partial spread.  Clearly, for an elementary pseudo-oval every induced spread is Desarguesian.
The following theorem shows that for $q$ odd, a strong version of the converse also holds.
\begin{theorem}{\rm \cite{Casse}}
If $\O$ is a pseudo-oval in $\PG(3n-1,q)$, $q$ odd, such that for at least one element the induced spread is Desarguesian, then $\O$ is a pseudo-conic.
\end{theorem}
The proof of this theorem relies on the theorem of Chen and Kaerlein \cite{Chen} for Laguerre planes in odd order, which in its turn relies on the theorem of Segre \cite{segre} characterising every oval in $\PG(2,q)$, $q$ odd, as a conic. This clearly rules out a similar approach for even characteristic. The characterisation of pseudo-ovals in terms of the induced spreads for even characteristic was posed as Problem A.3.4 in \cite{TGQ}.

In this paper, we will prove that the following holds:
\begin{rtheorem} If  $\O$ is a pseudo-oval in $\PG(3n-1,q)$, $q=2^h$, $h>1$, $n$ prime, such that the spread induced by every element of $\O$ is Desarguesian, then $\O$ is elementary.
\end{rtheorem}
As a corollary, we prove a similar statement for pseudo-hyperovals.
\begin{corollary} Let $\H$ be a pseudo-hyperoval in $\PG(3n-1,q)$, $q=2^h$, $h>1$, $n$ prime, such that the spread induced by at least $q^n+1$ elements of $\H$ is Desarguesian, then $\H$ is elementary.
\end{corollary}

\subsection{Field reduction, Desarguesian spreads and Segre varieties}\label{S1}

We recall the {\em Andr\'e/Bruck-Bose representation} of a translation plane of order $q^n$. Let $\S$ be a $(n-1)$-spread of the projective space $\Sigma_\infty = \PG(2n-1,q)$ and embed $\Sigma_\infty$ as hyperplane of $\PG(2n,q)$.
Consider the following incidence structure ${\mathcal A}(\S)=(\P,\mathcal{L})$, where incidence is natural:
\begin{itemize}\setlength{\itemsep}{-1pt}
\item[$\mathcal{P}:$] the points of $\PG(2n,q)\setminus \Sigma_\infty$ (the affine points),
\item[$\mathcal{L}:$] the $n$-spaces of $\PG(2n,q)$ intersecting $\Sigma_\infty$ exactly in an element of $\S$.
\end{itemize}

This defines an affine translation plane of order $q^n$ \cite{Andre,Br}. If the spread $\S$ is Desarguesian, $\mathcal{A}(\S)$ is a Desarguesian affine plane $\AG(2,q^n)$. Adding $\Sigma_\infty$ as the line at infinity, and considering the spread elements as its points, we obtain a projective plane of order $q^n$. 

An {\em $(n-1)$-regulus} or {\em regulus} $\R$ in $\PG(2n-1,q)$ is a set of $q+1$ mutually disjoint $(n-1)$-spaces having the property that if a line meets $3$ elements of $\R$, then it meets all elements of $\R$. There is a unique regulus through $3$ mutually disjoint $(n-1)$-spaces $A,B$ and $C$ in $\PG(2n-1,q)$, let us denote this by $\R(A,B,C)$. Every Desarguesian spread $\D$ has the property that for $3$ elements $A,B,C$ in $\D$, the elements of $\R(A,B,C)$ are also contained in $\D$, i.e. $\D$ is {\em regular} (see also \cite{Br}). Moreover, every Desarguesian spread $\D$ clearly has the property that the space spanned by $2$ elements of $\D$ is partitioned by elements of $\D$, i.e. $\D$ is {\em normal}.

We will use the following notation for points of a projective space $\PG(r-1,q^n)$. A point $P$ of $\PG(r-1,q^n)$ defined by a vector $(x_1,x_2,\ldots,x_r)\in (\F_{q^n})^r$ is denoted by $\F_{q^n}(x_1,x_2,\ldots,x_r)$, reflecting the fact that every $\F_{q^n}$-multiple of $(x_1,x_2,\ldots,x_r)$ gives rise to the point $P$.

An {\em $\F_{q^t}$-subline} in $\PG(1,q^n)$, where $t|n$, is a set of $q^t+1$ points in $\PG(1,q^n)$ that is $\PGL$-equivalent to the set $\{\F_{q^n}(1,x)|x\in \F_{q^t}\}\cup \{\F_{q^n}(0,1)\}$. As $\PGL(2,q^n)$ acts sharply $3$-transitively on the points of the projective line, we see that any $3$ points define a unique $\F_{q^t}$-subline.

%An {\em $\F_{q^t}$-subplane} of $\PG(2,q^n)$, where $t|n$, is a subgeometry $\PG(2,q^t)$ of $\PG(2,q^n)$, i.e. a set of $q^{2t}+q^t+1$ points and $q^{2t}+q^t+1$ lines in $\PG(2,q^n)$ forming an projective plane, where the point set is $\PGL$-equivalent to the set $\{\F_{q^n}(x_0,x_1,x_2)|(x_0,x_1,x_2)\in (\F_{q^t}\times\F_{q^t}\times\F_{q^t})^\ast\}$. A set of four points in general position forms a {\em frame} for $\PG(2,q^n)$ and any frame is $\PGL$-equivalent to $\{\F_{q^n}(1,0,0),\F_{q^n}(0,1,0),\F_{q^n}(0,0,1),\F_{q^n}(1,1,1)\}$. This implies that $4$ points in general position determine a unique $\F_{q^t}$-subplane containing these four points.

We can identify the vector space $(\F_{q})^{rn}$ with $(\F_{q^n})^r$, and hence, we can write every point of $\PG(rn-1,q)$ as $\F_q(x_1,x_2,\ldots,x_r)$, where $x_i\in \F_{q^n}$. In this way, by field reduction, a point $\F_{q^n}(x_1,x_2,\ldots,x_r)$ in $\PG(r-1,q^n)$ corresponds to the $(n-1)$-space $\F_{q^n}(x_1,x_2,\ldots,x_r)=\{\F_q(\alpha x_1,\alpha x_2,\ldots,\alpha x_r)|\alpha \in \F_{q^n}\}$ in $\PG(rn-1,q)$.

We will need a lemma on Desarguesian spreads which has a straightforward proof, but we include it for completeness.

\begin{lemma}\label{UniqueDesSpread}
Let $\D_1$ be a Desarguesian $(n-1)$-spread in a $(2n-1)$-dimensional subspace $\Pi$ of $\PG(3n-1,q)$, let $\mu$ be an element of $\D_1$ and let $E_1$ and $E_2$ be disjoint $(n-1)$-spaces disjoint from $\Pi$ such that $\langle E_1,E_2\rangle$ meets $\Pi$ exactly in the space $\mu$. Then there exists a unique Desarguesian $(n-1)$-spread of $\PG(3n-1,q)$ containing the elements of $\D_1$ and $\R(\mu,E_1,E_2)$.
\end{lemma}
\begin{proof} Since $\D_1$ is a Desarguesian spread in $\Pi$, we can choose coordinates for $\Pi$ such that $\D_1=\{\F_{q^n}(1,x)|x\in \F_{q^n}\}\cup \{\mu=\F_{q^n}(0,1)\}$. We embed $\Pi$ in $\PG(3n-1,q)$ by mapping a point $\F_q(x_1,x_2)$, $x_1,x_2\in \F_{q^n}$, of $\Pi$ to $\F_q(x_1,x_2,0)$. Consider a point $P$ of $\mu$ and let $\ell_P$ denote the unique transversal line through the point $P$ of $\mu$ to the regulus $\R(\mu,E_1,E_2)$.

We can still choose coordinates for $n+1$ points in general position in $\PG(3n-1,q)\setminus \Pi$. We will choose these $n+1$ points such that $n$ of them belong to $E_1$ and one of them belongs to $E_2$.
Consider a set $\{y_i|i=1,\ldots,n\}$ forming a basis of $\F_{q^n}$ over $\F_q$. We may assume that the line $\ell_{P_i}$ through $P_i=\F_q(0,y_i,0)$ meets $E_1$ in the point $\F_q(0,0,y_i)$. It follows that $E_1=\F_{q^n}(0,0,1)$. Moreover, we may assume that $\ell_Q$ with $Q=\F_q(0,\sum_{i=1}^n y_i,0)$ meets $E_2$ in $\F_q(0,\sum_{i=1}^n y_i,\sum_{i=1}^n y_i)$. Since $\F_q(0,\sum_{i=1}^n y_i,\sum_{i=1}^n y_i)$ has to be in the space spanned by the intersection points $R_i = \ell_{P_i} \cap E_2$, it follows that $R_i=\F_q(0,y_i,y_i)$ and consequently, that $E_2=\F_{q^n}(0,1,1)$.

It is clear that the Desarguesian spread $\D=\{\F_{q^n}(x_1,x_2,x_3)|x_1,x_2,x_3\in \F_{q^n}\}$ contains the spread $\D_1$ and the regulus $\R(\mu,E_1,E_2)$. Moreover, since a Desarguesian spread is normal, every element of $\D$, not in $\langle E_1,E_2\rangle$ is obtained as the intersection of $\langle E_1,X\rangle\cap \langle E_2,Y\rangle$, where $X,Y\in \D_1$, it is clear that $\D$ is the unique Desarguesian spread satisfying our hypothesis.
\end{proof}

\begin{theorem} \label{moeilijk} A set $\S$ of at least $3$ points in $\PG(1,q^n)$, $q>2$, such that any three points of $\S$ determine a subline entirely contained in $\S$, defines an $\F_{q^t}$-subline $\PG(1,q^n)$ for some $t|n$.
\end{theorem}
\begin{proof} 
Without loss of generality, we may choose the points $\F_{q^n}(0,1)$, $\F_{q^n}(1,0)$ and $\F_{q^n}(1,1)$ to be in $\S$. Put $S=\{x \ |\ \F_{q^n}(1,x)\in \S\}$, clearly $\F_q\subseteq S$.

Consider $x,y \in S$, where $x\neq y$ and $xy\neq 0$, then every point of the $\F_q$-subline through the distinct points $\F_{q^n}(0,1)$, $\F_{q^n}(1,x)$ and $\F_{q^n}(1,y)$ has to be contained in $\S$. The points of this subline, different from $\F_{q^n}(0,1)$ are given by $\F_{q^n}(1,x+(y-x)t)$, where $t\in \F_q$. This implies that if $x$ and $y$ are in $S$, also $(1-t)x+ty$ is in $S$ for all $t\in \F_q$. It easily follows that $S$ is closed under taking linear combinations with elements of $\F_q$, hence, $S$ forms an $\F_q$-subspace of $\F_{q^n}$.
%The subline through $\F_{q^n}(0,1)$, $\F_{q^n}(1,0)$ and $\F_{q^n}(1,1)$ is in $\S$, so every $x\in \F_q$ is contained in $\S$. 
%
%Now consider $z\notin \F_q$, and $y \in S$. The points of the $\F_q$-subline through $\F_{q^n}(1,0)$ and $\F_{q^n}(1,y)$ and $\F_{q^n}(1,z)$, different from $\F_{q^n}(1,y)$ are given by $\F_{q^n}((1+t)z-y,tyz)$, where $t\in \F_q$. Since $q>2$, we can consider an element $t\in\F_q$ such that $t(t-1)\neq 0$. Now, since $S$ is an $\F_q$-subspace, the element $s:=(t+1)z+1$ is an element of $\S$. Using $z\notin \F_q$, it is easy to check that $s\neq 0,z$. This implies that the point $\F_{q^n}((1+t)z-s,tys)=\F_{q^n}(1,tz((t+1)z+1))$ is contained in the $\F_{q}$-subline through $\F_{q^n}(1,0)$ and $\F_{q^n}(1,s)$ and $\F_{q^n}(1,z)$, meaning $t(t+1)z^2+tz\in S$. Since $tz\in S$ and $t(t+1)\neq 0$, it follows that $z^2\in S$.
%

%Consider $x,y \in S$, then every point of the $\F_q$-subline through the distinct points $\F_{q^n}(0,1)$, $\F_{q^n}(1,x)$ and $\F_{q^n}(1,y)$ has to be contained in $\S$. The points of this subline, different from $\F_{q^n}(0,1)$ are given by $\F_{q^n}(1,x+(y-x)t)$, where $t\in \F_q$. This implies that if $x$ and $y$ are in $S$, also $(1-t)x+ty$ is in $S$ for all $t\in \F_q$. It easily follows that $S$ is closed under taking linear combinations with elements of $\F_q$, hence, $S$ forms an $\F_q$-subspace.

Now consider $x',y'\in S$, $x',y'\neq 0$. We claim that (1) $x'^2/y'\in S$ and (2) $x'^2\in S$.

 If $y'/x'\in \F_q$, our claim (1) immediately follows from the fact that $S$ is an $\F_q$-subspace so we may assume that $y'/x'\in \F_{q^n}\setminus\F_q$. Since $q>2$, we can consider an element $t\in\F_q$ such that $t(t-1)\neq 0$. Put $z':=y'-(t-1)x'$. Since $S$ is an $\F_q$-subspace, $z'\in S$. It is easy to check that $z'\notin\{0,x'\}$. Every point of the $\F_q$-subline containing distinct points $\F_{q^n}(1,0)$, $\F_{q^n}(1,x')$ and $\F_{q^n}(1,z')$ has to be contained in $\S$, and the points of this subline, different from $\F_{q^n}(1,z')$, are given by $\F_{q^n}(z'-x'+t'x',tx'z')$, where $t'\in \F_q$. This implies that  $\frac{t'x'z'}{z'+(t'-1)x'}$ is in $S$ for every $t'\in \F_q$, so also for $t'=t$, which implies that   $tx'-\frac{t(t-1)x'^2}{y'}\in S$. Since $tx' \in S$ and $t(t-1)\neq 0$, we conclude that $\frac{x'^2}{y'} \in S$ which proves claim (1). Claim (2) follows immediately from Claim (1) by taking $y=1\in \F_q\subseteq S$.
%
%Consider $x,y\in S$, then every point of the $\F_q$-subline containing distinct points $\F_{q^n}(1,0)$, $\F_{q^n}(1,x)$ and $\F_{q^n}(1,y)$ has to be contained in $\S$, and the points of this subline, different from $\F_{q^n}(1,y)$, are given by $\F_{q^n}(y-x+tx,txy)$, where $t\in \F_q$. This implies that  for every $t\in \F_q$, $\frac{txy}{y+(t-1)x}$ is in $S$. Put $z=y+(t-1)x$, since $S$ is an $\F_q$-subspace and $x,y$ in $S$, we have $z\in S$. We rewrite $\frac{txy}{y+(t-1)x}$ as $tx-\frac{t(t-1)x^2}{z}$. Since $tx \in S$, it follows that $\frac{t(t-1)x^2}{z}$ is in $S$ for every $t \in \F_q$. Since $q>2$, we can consider an element $t\in\F_q$ such that $t(t-1)\neq 0$ and we conclude that $\frac{x^2}{z} \in S$.
%
%The previous argument applied for $z=1$ implies that if an element $x$ is in $S$, then also $x^2$ is in $S$.

Now let  $v,w\in S$ and first suppose that $q$ is odd, then $vw=\frac{1}{2}((v+w)^2-v^2-w^2)$, and since $S$ is an $\F_q$-subspace and by claim (2), all terms on the right hand side are in $S$, so is $vw$. If $q$ is even, say $q^n=2^h$, then  $v=u^2$ for some $u\in \F_{q^n}$, but since $u=u^{2^h}=v^{2^{h-1}}$, $v$ is contained in $S$. This implies that $\frac{v}{w}=\frac{u^2}{w}\in S$ by claim (1) and consequently, again by claim (1), $vw=\frac{v^2}{v/w}\in S$. In both cases, we get that $S$ is a subfield of $\F_{q^n}$ and the statement follows.
\end{proof}

\begin{corollary}\label{regpluseen}Let $\D_1$ and $\D_2$ be two Desarguesian $(n-1)$-spreads in $\PG(2n-1,q)$, $q=p^h$, $p$ prime, $q>2$, with at least $3$ elements in common, then $\D_1$ and $\D_2$ share exactly $q^t+1$ elements for some $t|n$. In particular, if $n$ is prime, then $\D_1$ and $\D_2$ share a regulus or coincide.
\end{corollary}
\begin{proof} Let $X$ be the set of common elements of $\D_1$ and $\D_2$. Since a Desarguesian spread $\D$ is regular, it has to contain the regulus defined by any three elements of $\D$, which, since $\D_1$ and $\D_2$ are Desarguesian, implies that the regulus through $3$ elements of $X$ is contained in $X$. Now since $X$ is contained in a Desarguesian spread, $X$ corresponds to a set of points $\S$ in $\PG(1,q^n)$ such that every $\F_q$-subline through $3$ points of $\S$ is contained in $\S$. The first part of the statement now follows from Theorem \ref{moeilijk}. The second part follows from the fact that the only divisors of a prime $n$ are $1$ and $n$.
\end{proof}

An {\em $\F_{q}$-subplane} of $\PG(2,q^n)$, is a subgeometry $\PG(2,q)$ of $\PG(2,q^n)$, i.e. a set of $q^{2}+q+1$ points and $q^{2}+q+1$ lines in $\PG(2,q^n)$ forming an projective plane, where the point set is $\PGL$-equivalent to the set $\{\F_{q^n}(x_0,x_1,x_2)|(x_0,x_1,x_2)\in (\F_{q}\times\F_{q}\times\F_{q})\setminus(0,0,0)\}$. If we apply field reduction to the point set of an $\F_q$-subplane, we find a set $\S$ of $q^2+q+1$ elements of a Desarguesian spread $\D$. All elements of $\S$ meet a fixed plane of $\PG(3n-1,q)$ and form one system of a Segre variety $\mathbf{S}_{n-1,2}$ (see e.g. \cite{FQ11}). Note that $\mathbf{S}_{n-1,2}$ is contained in $\PG(3n-1,q)$ and consists of two systems of subspaces, one with subspaces of dimension $(n-1)$ and the other consisting of planes. Moreover, every point of $\mathbf{S}_{n-1,2}$ lies on exactly one subspace of each system.

As $\PGL(3,q^n)$ acts sharply transitively on the frames of $\PG(2,q^n)$, we see that $4$ points in general position define a unique $\F_q$-subplane of $\PG(2,q^n)$. A similar statement holds for $4$ $(n-1)$-spaces in $\PG(3n-1,q)$ in general position. A proof can be found in e.g. \cite[Proposition 2.1, Corollary 2.3, Proposition 2.4]{corrado}.

\begin{lemma} \label{vier} Four $(n-1)$-spaces in $\PG(3n-1,q)$ in general position are contained in a unique Segre variety $\mathbf{S}_{n-1,2}$. \end{lemma}

\section{Laguerre planes}\label{S3}
\begin{definition}
A {\em Laguerre plane} is an incidence structure with points $\P$, lines $\L$ and circles $\C$ such that $(\P,\L,\C)$ satisfies the following four axioms:

\begin{itemize}\setlength{\itemsep}{-1pt}
\item[AX1] Every point lies on a unique line.
\item[AX2] A circle and a line meet in a unique point.
\item[AX3] Through $3$ points, no two collinear, there is a unique circle of $\C$.
\item[AX4] If $P$ is a point on a fixed circle $C$ and $Q$ a point, not on the line through $P$ and not on the circle $C$, then there is a unique circle $C'$ through $P$ and $Q$, meeting $C$ only in the point $P$.
\end{itemize}
\end{definition}

In a finite Laguerre plane, every circle contains $s+1$ points for some $s$; this constant $s$ is called the {\em order} of the Laguerre plane.

Starting from a point $P$ of a Laguerre plane $\Ll=(\P,\L,\C)$, we obtain an affine plane $(\P',\L')$, where incidence is inherited from $\Ll$, as follows.

\begin{itemize}\setlength{\itemsep}{-1pt}
\item[$\P':$] the points of $\P$, different from $P$ and not collinear with $P$,
\item[$\L':$] (1) the lines of $\L$ not through $P$,\\
(2) the elements of $\C$ through $P$.
\end{itemize}
The obtained affine plane $(\P',\L')$ is called the {\em derived affine plane} at $P$.

%Note that an incidence structure $(\P,\L,\C,\I)$ satisfying AX1, AX2, AX3, such that $|\P|= s(s+1)$, $|\L|=s+1$, $|\C|= s^3$, for some $s$, always satisfies AX4, and is therefore a Laguerre plane.
\begin{definition}
A finite {\em ovoidal} Laguerre plane with points $\P$, lines $\L$ and circles $\C$ is a Laguerre plane that can be constructed from a cone $\K$ as follows. Consider a cone $\K$ in $\PG(3,q)$ with vertex the point $V$ and base an oval in a plane $H$, not containing $V$. Incidence is natural.
\begin{itemize}\setlength{\itemsep}{-1pt}
\item[$\mathcal{P}:$] the points of $\K\setminus\{V\}$,
\item[$\mathcal{L}:$] the {\em generators} of $\K$, i.e. the lines of $\PG(3,q)$, lying on $\K$,
\item[$\mathcal{C}:$] the plane sections of $\K$, not containing $V$.
\end{itemize}
\end{definition}

For later use, we will consider the dual model in $\PG(3,q)$ of the definition of an ovoidal Laguerre plane obtained from the cone $\K$ with vertex $V$ and base an oval $A$, embedded in $\PG(3,q)$. Let $H$ denote the plane which is the dual of the point $V$ in $\PG(3,q)$. Let $\overline{A}$ denote the dual (in $\PG(2,q)$) of the oval $A$ contained in $H$. It is not hard to see that we find the following incidence structure $(\P, \L,\C)$:
\begin{itemize}\setlength{\itemsep}{-1pt}
\item[$\P:$] planes different from $H$ and meeting $H$ in a line of $\overline{A}$,
\item[$\L:$] the lines in $H$ belonging to $\overline{A}$,
\item[$\C:$] the points of $\PG(3,q)$ not contained in $H$ (the affine points).
\end{itemize} 
We will denote the ovoidal Laguerre plane that is obtained in this way by $L(\overline{A})$.

\begin{definition}
The {\em classical} Laguerre plane of order $q$ is an ovoidal Laguerre plane, obtained from a quadratic cone $\K$ in $\PG(3,q)$, i.e. a cone whose base is a conic.
\end{definition}
\begin{remark}
A Laguerre plane is called {\em Miquelian} if for each eight pairwise different points $A,B,C,D,E, F,G,H$ it follows from $(ABCD)$,
$(ABEF)$, $(BCFG)$, $(CDGH)$, $(ADEH)$ that $(EFGH)$, where $(PQRS)$ denotes that $P,Q,R,S$ are on a common circle. By a theorem of van der Waerden and Smid a Laguerre plane is Miquelian if and only if it is classical \cite{vs} and we, as well as many others, use the term `Miquelian Laguerre plane' instead of `classical Laguerre plane'.
\end{remark}
It follows from Segre's theorem that an ovoidal Laguerre plane of odd order is necessarily Miquelian.

For later use, we will also introduce the {\em plane model} of the Miquelian Laguerre plane of even order $q$ (for more information we refer to \cite{LaguerrePlane}). Consider a point $N$ in $\PG(2,q)$, $q$ even. Since three points together with a nucleus determine a unique conic, one can easily count that there are exactly $q^{3}-q^{2}$ conics in $\PG(2,q)$, $q$ even, all having the same point $N$ as their nucleus. The plane model of the Miquelian Laguerre plane is the following incidence structure $(\P, \L, \C)$ embedded in $\PG(2,q)$, $q$ even, with natural incidence.
\begin{itemize}\setlength{\itemsep}{-1pt}
\item[$\P:$] the points of $\PG(2,q)$ different from $N$,
\item[$\L:$] the lines of $\PG(2,q)$ containing $N$,
\item[$\C:$] the $q^{2}$ lines of $\PG(2,q)$ not containing to $N$ and the $q^{3}-q^{2}$ conics in $\PG(2,q)$ having $N$ as their nucleus.
\end{itemize}
%\begin{remark}
%One can easily deduce this model from the standard cone model in the following way. Consider the cone $K$ with vertex $V$ and as base a conic $A$ in the plane $H_\infty$. Take a point $R$ on the line $NV$, different from $N$ and $V$, where $N$ is the nucleus of $A$. Consider the projection of the cone $K$ from $R$ onto $H_\infty$. Distinct points of the cone are mapped to distinct points of the plane $H_\infty$. The $q^2$ plane sections of $\K$ containing $R$ are mapped to distinct lines not containing $N$. The $q^3-q^2$ plane sections not containing $R$ are mapped to distinct conics of $H_\infty$. Moreover, all these conics have $N$ as their nucleus. Clearly, this projection is incidence preserving.
%\end{remark}
\begin{remark}
One can easily deduce this model from the standard cone model obtained from a quadratic cone $\K$ with vertex $V$ and base a conic $\C$ by projecting the cone $\K$ from a point on the line through $V$ and the nucleus of $\C$ on a plane.
\end{remark}

The {\em kernel} $K$ of a Laguerre plane $\Ll$ is the subgroup of $\Aut(\Ll)$ consisting of all automorphisms which map a point $P$ onto a point collinear with $P$, for every point $P$ of $\Ll$. In other words, $K$ is the elementwise stabiliser of lines of $\Ll$.
\begin{lemma}{\rm (see e.g. \cite[Theorem 1]{Steinke})} \label{St}The order of the kernel $K$ of a Laguerre plane $\Ll$ of order $s$ divides $s^3(s-1)$. Moreover, $|K|=s^3(s-1)$ if and only if $\Ll$ is ovoidal.
\end{lemma}

\begin{definition}
A Laguerre plane $\Ll$ is an {\em elation Laguerre plane} if its kernel $K$ acts transitively on the circles of $\Ll$.
\end{definition}
We denote the dual of a subspace $M$ or a set of subspaces $\O$ of $\PG(3n-1,q)$ by $\overline{M}$ and $\overline{\O}$.

A dual pseudo-oval $\overline{\O}$  in $\PG(3n-1,q)$ gives rise to an elation Laguerre plane $L(\overline{\O})$ in the following way. Embed $H_\infty=\PG(3n-1,q)$ as a hyperplane in $\PG(3n,q)$ and define $L(\overline{\O})$ to be the incidence structure $(\P,\L,\C)$ with natural incidence and:
\begin{itemize}\setlength{\itemsep}{-1pt}
\item[$\P:$] $2n$-spaces meeting $H_\infty$ in an element of $\overline{\O}$,
\item[$\L:$] elements of $\overline{\O}$,
\item[$\C:$] points of $\PG(3n,q)$ not in $H_\infty$ (the affine points).
\end{itemize}
It is not hard to check that this incidence structure defines a Laguerre plane of order $q^n$ and that the group of perspectivities with axis $H_\infty$ in $\PGammaL(3n,q)$ induces a subgroup of the kernel of $L(\overline{\O})$ that acts transitively on the circles of $L(\overline{\O})$. So $L(\overline{\O})$ is indeed an elation Laguerre plane.

%\begin{definition}
%A Laguerre plane is called an {\em elation} Laguerre plane if it admits a group of automorphisms which fixes each line and acts regularly on the set of circles.
%\end{definition}

%\begin{lemma}\cite{Steinke}\label{ovoidal} An elation Laguerre plane with order $s$, has kernel $K$ of order $s^3(q-1)$ for some prime power $q$, and $s=q^n$ for some $n$. A Laguerre plane of order $s$ is ovoidal if and only if $|K|=s^3(s-1)$.
%\end{lemma}

In \cite{Steinke}, Steinke showed the converse: every elation Laguerre plane can be constructed from a dual pseudo-oval.
\begin{theorem}{\rm\cite{Steinke}}\label{Steinke} A finite Laguerre plane $\Ll$ is an elation Laguerre plane if and only if $\Ll\cong L(\overline{\O})$ for some dual pseudo-oval $\overline{\O}$.
\end{theorem}
More explicitely, it is shown that a Laguerre plane of order $q^n$ with kernel of order $q^{3n}(q-1)$ can be obtained from a dual pseudo-oval in $\PG(3n-1,q)$.

We show in Theorem \ref{hoofd} that every elementary dual pseudo-oval gives rise to an ovoidal Laguerre plane and vice versa. In order to prove this, we need the following lemma.

%\begin{proof} By \cite{Maurer}, all automorphisms of an ovoidal Laguerre plane are induced by automorphisms of $\PG(3,q^n)$ . Hence, all elements in the stabiliser of the lines of $L$ correspond to elements of $\PGammaL(4,q^n)$, fixing all lines of the cone $K$, hence, fixing $V$ and all lines through $V$. The number of elements fixing all lines through a given point equals the number of collineations with a fixed axis in $\PG(3,q^n)$, which equals $q^{3n}(q^n-1)$. Moreover, it is clearly not possible for two different elements of the line-wise stabiliser of the point $V$ to induce the same automorphism of the cone $K$.
%\end{proof}
%In \cite{Steinke}, it is shown that the converse statement also holds.
\begin{lemma} \label{unique} Let $\Ll$ be an ovoidal Laguerre plane of order $q^n$, then there is a unique subgroup $T$ of order $q^{3n}$ in the kernel $K$ of $\Ll$.
\end{lemma}
\begin{proof} Consider the dual model for an ovoidal Laguerre plane. Every perspectivity in $\PGammaL(4,q^n)$ with axis $H_\infty$ induces an element of $K$. Since the group of perspectivities with axis $H_\infty$ has order $q^{3n}(q^n-1)$, which equals the order of $K$ by Lemma \ref{St}, it follows that every element of $K$ corresponds to a perspectivity. The group $G_{el}$ consisting of all elations in $\PG(3,q^n)$ with axis $H_\infty$ is a normal subgroup of the group of all perspectivities with axis $H_\infty$ and has order $q^{3n}$.

Let $S$ be a subgroup of $K$ of order $q^{3n}$, $q=p^h$, $p$ prime, then $S$ is a Sylow $p$-subgroup and since all Sylow $p$-subgroups are conjugate and $G_{el}$ is normal in $K$, $S=G_{el}$.
\end{proof}

\begin{theorem}\label{hoofd} A finite elation Laguerre plane $\Ll$ is ovoidal if and only if $\Ll\cong L(\overline{\O})$ where $\overline{\O}$ is an elementary dual pseudo-oval in $\PG(3n-1,q)$.
\end{theorem}
\begin{proof} Let $\Ll$ be an elation Laguerre plane. By Theorem \ref{Steinke}, $\Ll$ is isomorphic to $L(\overline{\O})$, where $\overline{\O}$ is a dual pseudo-oval in $\PG(3n-1,q)$, for some $q$ and $n$ such that the order of $\Ll$ is $q^n$. So it remains to show that $L(\overline{\O})$ is ovoidal if and only if $\overline{\O}$ is elementary. In view of the definition of an ovoidal Laguerre plane, using the dual setting, we will show that $L(\overline{\O})$ is isomorphic to $L(\overline{A})$ if and only if the dual pseudo-oval $\overline{\O}$ in $\PG(3n-1,q)$ is obtained from the dual oval $\overline{A}$ in $\PG(2,q^n)$ by field reduction.

First suppose that the dual pseudo-oval $\overline{\O}$ in $\PG(3n-1,q)$ is obtained from a dual oval, say $\overline{A}$, in $\PG(2,q^n)$ by field reduction. Apply field reduction to the points, lines and circles of $L(\overline{A})$, then the obtained incidence structure $\Ll^*$, contained in $\PG(4n-1,q)$ is isomorphic to $L(\overline{A})$. If we intersect the points, lines and circles of $\Ll^*$ with a fixed $3n$-dimensional subspace of $\PG(4n-1,q)$, through the $(3n-1)$-space containing the field reduced elements of $\overline{A}$, then the obtained structure is clearly isomorphic to the points, lines and circles from $L(\overline{\O})$.

Now, let $\Ll=(\P,\L,\C)$ be a Laguerre plane that on the one hand is isomorphic to $L(\overline{\O})$ (call this {\em model 1}) and on the other hand isomorphic to $L(\overline{A})$ (call this {\em model 2}). As before, the elementwise stabiliser of the lines in the automorphism group $\Aut(\Ll)$ of $\Ll$ (the kernel of $\Ll$) is denoted by $K$.

From model 1, we know that the group of elations in $\PG(3n,q)$, with axis the hyperplane $H_\infty$ which contains the elements of $\overline{\O}$, induces a subgroup of $K$ of order $q^{3n}$, likewise, from model 2, we know that the group of elations in $\PG(3,q^n)$ with axis the hyperplane $H$ which contains the elements of $\overline{A}$ induces a subgroup of $K$ of order $q^{3n}$. By Lemma \ref{unique} these induced subgroups are the same, denote this group by $T$. Consider the stabiliser of a point $P$ in $T$. From model 2, we have that $T_P$ has order $q^{2n}$, the number of elations with axis $H$ fixing a plane of $\PG(3,q^n)$, intersecting $H$ in a line of $\overline{A}$. In model 1, the elements of $T_P$ correspond to elations of $\PG(3n-1,q)$ fixing a $2n$-space intersecting $H_\infty$ in an element of $\overline{\O}$.

The group $T$ corresponds to the elations in $\PG(3,q^n)$ (model 1), hence $T$ forms a $3$-dimensional vector space over $\F_{q^n}$. Equivalently, the group $T$ corresponds to the elations in $\PG(3n-1,q)$ (model 2), hence also forms a $3n$-dimensional vector space over $\F_{q}$. Since $T_P$ in both models is normalised by the perspectivities, we see that $T_P$ forms a $2$-dimensional vector subspace $W=V(2,q^n)$ (model 1) and a $2n$-dimensional vector subspace $W'=V(2n,q)$ (model 2) (see also \cite{Penttila}). Clearly, since $W$ and $W'$ correspond to the same vector space, $W'$ is obtained from $W$ by field reduction. Choose for every line $\ell_i$ of $L$, one point $P_i\in \ell_i$. Since a point $P_i$ lies on a unique line $\ell_i$ of $L$, $T_{P_i}$ can be identified with the line $\ell_i$. Considering this projectively, we get that for all $i=1,\ldots,q^n+1$, the subgroup $T_{P_i}$, which forms a $2$-dimensional vector space over $\F_{q^n}$ and a $2n$-dimensional vector space over $\F_{q}$, is identified on one hand to an element of $\overline{\O}$ (model 1) and on the other hand to a line of $\overline{A}$ (model 2). This implies that $\overline{\O}$ is obtained from $\overline{A}$ by field reduction.
 \end{proof}

From this we can easily deduce the following corollaries.
\begin{corollary}\label{Miquelian is conic}
A finite elation Laguerre plane $\Ll$ is Miquelian if and only if $\Ll\cong L(\overline{\O})$ where $\overline{\O}$ is a dual pseudo-conic in $\PG(3n-1,q)$.
\end{corollary}

\begin{corollary}\label{lem2}Let $\overline{\H}$ be a dual pseudo-hyperoval containing an element $\overline{E}$ such that $L(\overline{\O})$, where $\overline{\O}=\overline{\H}\setminus \overline{E}$, is Miquelian, then $\H$ is a pseudo-hyperconic with $E$ as the field reduced nucleus.
\end{corollary}
\begin{proof}
By Corollary \ref{Miquelian is conic}, $\overline{\O}$ is obtained by applying field reduction to a dual conic $\overline{\C}$ in $\PG(2,q^n)$. The dual conic $\overline{\C}$ in $\PG(2,q^n)$ uniquely extends to a dual hyperconic by adding its dual nucleus line $\overline{N}$. This shows that $\overline{\O}$ can be extended to a dual pseudo-hyperoval by the $(2n-1)$-space obtained by applying field reduction to the line $\overline{N}$. Since Theorem \ref{Thas PHO} shows that this extension is unique, we see that the element $E$ is the $(n-1)$-space obtained by applying field reduction to the nucleus $N$ of the conic $\C$, and hence, $\H$ is a pseudo-hyperconic.
\end{proof}

\section{Towards the proof of the main theorem} \label{S4}

Recall that we will prove the following:

\begin{rtheorem} If  $\O$ is a pseudo-oval in $\PG(3n-1,q)$, $q=2^h$, $h>1$, $n$ prime, such that the spread induced by every element of $\O$ is Desarguesian, then $\O$ is elementary.
\end{rtheorem}
We know from Theorem \ref{Thas PHO} that a pseudo-oval $\O$ in even characteristic extends in a unique way to a pseudo-hyperoval $\H$ and for the proof of our main theorem, we will work with $\H$, the unique pseudo-hyperoval extending $\O$.

%As said before, we will only restrict to the case where $n$ is a prime to finish our proof.
%OPMERKING?
%We know from Theorem \ref{Thas PHO} that $\O$ extends in a unique way to a pseudo-hyperoval $\H$. Note that if a pseudo-oval $\O$ has the property that the (partial) spread induced by every element of $\O$ is Desarguesian, we have that for the $q^n+1$ elements corresponding to $\O$ of $\H$, the spread induced is Desarguesian because a partial spread of size $q^n+1$ extends in a unique way to a spread.
%
%
%\begin{definition} We will call an $\FR_{q^t}$-subplane in a Desarguesian spread of $\PG(3n-1,q)$ {\bf crowded} (w.r.t. a pseudo-hyperoval $\H$) if it contains $q^t+2$ elements of $\H$.
%\end{definition}
%Note that for $t=n$ the existence of a crowded $\FR_{q^n}$-subplane implies that $\H$ is contained in a Desarguesian spread, hence, is elementary.

We will split the proof of the Main Theorem in two cases.
In Subsection \ref{easy} we will consider pseudo-hyperovals having a specific property (P1) and we will prove that they are always elementary. In Subsection \ref{difficult} we will consider dual pseudo-hyperovals satisfying a property (P2), and again we show that they are elementary. Finally, in Subsection \ref{conclusion} we see that if a pseudo-oval $\O$, such that every element induces a Desarguesian spread, extends to a pseudo-hyperoval $\H$ which does not meet property (P1), then its dual $\overline{\H}$ necessarily meets (P2), which implies that $\O$ is elementary.

\subsection{Case 1}\label{easy}
In this subsection, we will consider a pseudo-hyperoval $\H$ having the following property:
\begin{itemize}
\item[(P1):] there exist four elements $E_i$, $i=1,\ldots,4$ of $\H$, such that
\begin{itemize}
\item[(i)] the induced spreads $\S_1$, $\S_2$, $\S_3$ are Desarguesian,
\item[(ii)] the unique $\mathbf{S}_{n-1,2}$ through $E_1,E_2,E_3$ and $E_4$ does not contain $q+2$ elements of $\H$.
\end{itemize}
\end{itemize}
%\begin{remark}
%Note that the Desarguesian spread $\D$ in Property (P1) is uniquely determined by Lemma \ref{UniqueDesSpread}.
%\end{remark}

\begin{theorem}\label{case1} Consider a pseudo-hyperoval $\H$ in $\PG(3n-1,q)$, $q=2^h$, $h>1$, $n$ prime, satisfying Property (P1), then $\H$ is elementary.
\end{theorem}
\begin{proof}
Let $E_1,\ldots,E_4$ be the four elements obtained from the hypothesis that $\H$ satisfies Property (P1). Denote the $(n-1)$-space $\langle E_1,E_2\rangle\cap \langle E_3,E_4\rangle$ by $\mu$. %By Lemma \ref{UniqueDesSpread}, the Desarguesian spread $\D$ containing $\S_{1}$ and $\R(\mu,E_1,E_2)$ mentioned in property (P1) is unique.
The spreads $\S_{1}$ and $\S_{2}$ can be seen in $\langle E_3,E_4\rangle=\PG(2n-1,q)$. By Property (P1), $\S_1$ and $\S_2$ are Desarguesian. Since by definition $E_3,E_4$ and $\mu$ are contained in $\S_{1}$ and $\S_{2}$,  and $\S_1$ and $\S_2$ are Desarguesian and hence regular, the $q+1$ elements of the unique regulus $\R(\mu,E_3,E_4)$ through $E_3,E_4$ and $\mu$ are contained in $\S_{1}$ and $\S_{2}$. We claim that $\S_1=\S_2$.

We see that $\mu,E_1,E_2$ are elements of the spread $\S_{3}$ considered in $\langle E_1,E_2\rangle$. By Property (P1), $\S_3$ is Desarguesian, hence, regular, so every element of $\R(\mu,E_1,E_2)$ is contained in $\S_{3}$. Because $q>2$, we may take an element $X$ of $\R(\mu,E_1,E_2)$, different from $E_1,E_2$ and $\mu$.

Since $X\in \S_3$, the space $\langle X, E_3\rangle$ contains an element, say $E_5$, of $\H$. The $(2n-1)$-space $\langle E_1,E_5\rangle$ meets $\langle E_3,E_4\rangle$ in an $(n-1)$-space $Y$, that is by construction contained in $\S_{1}$.
Let $\D$ be the unique Desarguesian spread obtained from Theorem \ref{UniqueDesSpread}, through $\S_1$ and $E_1,E_2$. Since $E_5=\langle X,E_3\rangle \cap \langle Y,E_1\rangle$ and a Desarguesian spread is normal, we see that $E_5\in \D$. This holds for every element $E_i \in \H$ contained in $\langle Z, E_3\rangle$ with $Z\in \R(\mu,E_1,E_2)$; let $E_5,\ldots,E_{q+2}$ be these elements of $\H$.

Now consider the $(n-1)$-spaces $T_i:=\langle E_2,E_i\rangle\cap \langle E_3,E_4\rangle$, with $i=5,\ldots,q+2$. The spaces $T_i$ by definition belong to $\S_{2}$ (considered in $\langle E_3,E_4\rangle$). But since $E_2,E_i,E_3,E_4$ are elements of $\D$, $T_i$ is an element of $\D$ and since $\D\cap \langle E_3,E_4\rangle=\S_1$, $T_i\in \S_1$.

So the spreads $\S_{1}$ and $\S_{2}$ contain $\R(\mu,E_3,E_4)$ and all elements $T_i$. Suppose that all elements $T_i$, $i=5,\ldots,q+2$ are contained in $\R(\mu,E_3,E_4)$. 
Let $P$ be a point of $\mu$, let $\ell$ be the unique transversal line through $P$ to the regulus $\R(\mu,E_1,E_2)$ and let $m$ be the unique transversal line through $P$ to the regulus $\R(\mu,E_3,E_4)$. It is clear that the plane $\langle \ell,m\rangle$ is a plane of the second system of the unique $\mathbf{S}_{n-1,2}$, say $\B$, through $E_1,E_2,E_3,E_4$. This implies that all elements $T_i$, as well as the elements of $\R(\mu,E_1,E_2)$ are contained in $\B$.

The element $E_i$, $i=5,\ldots,q+2$ is obtained as $\langle T_i,E_2\rangle\cap \langle Z,E_3\rangle$, for some $Z\in \R(\mu,E_1,E_2)$. Now it is clear that $\mathbf{S}_{n-1,2}$ has the property that an $(n-1)$-space that is obtained as the intersection of the span of two elements of $\mathbf{S}_{n-1,2}$ is contained in $\mathbf{S}_{n-1,2}$. Since $T_i,E_2,Z,E_3$ are $(n-1)$-spaces of $\B$, $E_i$ is in $\B$, for all $i=1,\ldots,q+2$. This implies that $\B$ contains $q+2$ elements of $\B$, a contradiction since $\H$ satisfies Property (P1). 

Since $\S_1$ and $\S_2$ have more elements in common than the elements of the regulus $\R(\mu,E_3,E_4)$, using the fact that $n$ is prime, we see that Corollary \ref{regpluseen} proves our claim.

  Since $\S_1=\S_2$, every element $E$ of $\H$, different from $E_1,E_2,E_3,E_4$ can be written as $\langle E_1,U\rangle \cap \langle E_2,V\rangle$, where $U,V$ are elements of $\S_1=\S_2$. Since the Desarguesian spread $\D$ is normal, it follows that $E\in \D$ for all $E\in \H$. Since $\H$ is contained in a Desarguesian spread, $\H$ is elementary.
\end{proof}

\subsection{Case 2}\label{difficult}
%Consider as before a pseudo-hyperoval $\H$ in $\PG(3n-1,q)$, $q=2^h$, $h>1$, such that there are $q^n+1$ elements of $\H$ such that the spread induced by every element of $\H$ is Desarguesian.
%We will proceed in this case explicitely assuming that the following conjecture holds true.
%\begin{conjecture}\label{conj1} Let $\C$ be an oval in $\PG(2,q^n)$, $q>2$ even, and let $N$ be the unique point extending $\C$ to a hyperoval. Suppose that for every triple of distinct points $P_1,P_2,P_3$, there is a divisor $t<n$ of $n$ such that the $\F_{q^t}$-subplane through $P_1,P_2,P_3$ and $N$ contains $q^t+1$ elements of $\C$, then $\C$ is a conic with nucleus $N$.
%\end{conjecture}
%Note that for different choices of triples $P_1,P_2,P_3$, the obtained value of $t$ is allowed to be different. When the value of $t$ is constant for all choices of the triple $P_1,P_2,P_3$, then the conjecture follows from the following result. In particular, taking into account that the only divisor $t<n$ of a prime number $n$ is equal to $1$, this conjecture holds for $n$ prime.
In this subsection, we will use the following theorem on hyperovals.
\begin{theorem}{\rm\cite[Theorem 11, Remark 5]{PenttilaInversivePlanes}}\label{Tim}
Let $\O$ be an oval of $\PG(2,q^n)$, $q>2$ even. Let $N$ be the unique point extending $\O$ to a hyperoval. Then $\O$ is a conic if and only if every triple of distinct points of $\O$ together with $N$ lie in an $\F_q$-subplane that meets $\O$ in $q+1$ points.
\end{theorem}

In the proof of this case we will work in the dual setting, so we need the following lemma on dual pseudo-(hyper)ovals. 
\begin{lemma}\label{dualdes} Let $\O$ be a pseudo-oval in $\PG(3n-1,q)$ such that every element $E_i\in \O$, $i=1,\ldots,q^n+1$ induces a Desarguesian spread $\S_{i}$, then the dual pseudo-oval $\overline{\O}$ has the property that for every element $\overline{E_i}$, the set of intersections $\{\overline{E_j}\cap \overline{E_i}|j\neq i\}$ forms a partial spread in $\overline{E_i}$ uniquely extending to a Desarguesian spread and vice versa. The analoguous statement holds for pseudo-hyperovals.
\end{lemma}
\begin{proof}  An element of $\S_{i}$, say $E_1/E_i$ equals $\langle E_1,E_i\rangle/E_i$. This space can be identified with $\langle E_1,E_i\rangle$ and its dual $\overline{\langle E_1,E_i\rangle}$, which equals $\overline{E_1}\cap \overline{E_i}$. This implies that the set $\{E_1,\ldots,E_{i-1},E_{i+1},\ldots,E_{q^n+1}\}/E_i$  extends to a Desarguesian spread of $\PG(2n-1,q)$ if and only if $\{\overline{E_1}\cap \overline{E_{i}},\ldots,\overline{E_{i-1}}\cap \overline{E_{i}}, \overline{E_{i+1}}\cap \overline{E_{i}}, \ldots, \overline{E_{q^n+1}}\cap \overline{E_i}\}$ extends to a Desarguesian spread. The same reasoning holds for pseudo-hyperovals.
\end{proof}
By abuse of notation, we say that an element $\overline{E_i}$ of a dual pseudo-hyperoval $\overline{\H}=\{\overline{E_1},\ldots,\overline{E_{q^n+2}}\}$ {\em induces} the spread $\overline{\S}_i:=\{\overline{E_j}\cap \overline{E_i}| j\neq i\}$. Then Lemma \ref{dualdes} states that $\S_i$ is Desarguesian if and only if $\overline{\S}_i$ is Desarguesian. Also, we write $\overline{\mathbf{S}_{n-1,2}}$ for the set of $(2n-1)$-spaces in $\PG(3n-1,q)$ that is obtained by dualising the system of $(n-1)$-spaces of $\mathbf{S}_{n-1,2}$. In the case that $n=3$, both systems have spaces of dimension $2$, so we dualise the system of planes that contains the elements $E_1,E_2,E_3,E_4$ used to define the Segre variety $\mathbf{S}_{2,2}$.

We know that the $(n-1)$-spaces of $\mathbf{S}_{n-1,2}$ correspond to the points of an $\F_{q}$-subplane $\pi$ of $\PG(2,q^n)$, and are exactly the elements of a Desarguesian spread meeting a fixed plane. By considering the field reduction of the lines of the $\F_q$-subplane $\pi$ we can also see that $\overline{\mathbf{S}_{n-1,2}}$ consists of $q^2+q+1$ $(2n-1)$-spaces in $\PG(3n-1,q)$ each meeting a fixed plane in a different line of this plane.

 Suppose now the dual pseudo-hyperoval $\overline{\H}$ has an element $\overline{E_1}$ such that $\overline{E_1}$ and $\overline{\H}$ satisfy the following properties:
\begin{itemize}
\item[(P2):] \begin{itemize}
\item[(i)] $\overline{E_1}$ induces a Desarguesian spread,
\item[(ii)] for any three elements $\overline{E_2},\overline{E_3},\overline{E_4}$ of $\overline{\H}\setminus \{\overline{E_1}\}$, the unique $\overline{\mathbf{S}_{n-1,2}}$ through $\overline{E_1},\overline{E_2},\overline{E_3}$ and $\overline{E_4}$ contains $q+2$ elements of $\overline{H}$. \end{itemize}
\end{itemize}

%\begin{remark}
%Note that the Desarguesian spread $\D'$ in Property (P2) is uniquely determined by Lemma \ref{UniqueDesSpread}.
%\end{remark}

%We will proceed in this case explicitely assuming that the following conjecture holds true. This will result in a full proof for the case that $n$ is prime.
%\begin{conjecture}\label{conj1} Let $\C$ be an oval in $\PG(2,q^n)$, $q>2$ even, and let $N$ be the unique point extending $\C$ to a hyperoval. Suppose that that for every triple of distinct points $P_1,P_2,P_3$, there is a divisor $t<n$ of $n$ such that the $\F_{q^t}$-subplane through $P_1,P_2,P_3$ and $N$ contains $q^t+1$ elements of $\C$, then $\C$ is a conic with nucleus $N$.
%\end{conjecture}
%For $n$ is prime, the truth of this conjecture follows immediately from the following result, taking into account that the only divisor $t<n$ of $n$ equals $1$.
%\begin{theorem}{\rm\cite[Theorem 11, Remark 5]{PenttilaInversivePlanes}}\label{Tim}
%Let $\O$ be an oval of $\PG(2,q^n)$, $q>2$ even. Let $N$ be the unique point extending $\O$ to a hyperoval. Then $\O$ is a conic if and only if every triple of distinct points of $\O$ together with $N$ lie in a $\F_q$-subplane that meets $\O$ in $q+1$ points.
%\end{theorem}

%In this section we will prove that a pseudo-hyperoval $\H$ in $\PG(3n-1,q)$, $q=p^h$, $p$ and $n$ prime, $q>2$, such that the spread induced by every element of $\H$ is Desarguesian is a pseudo-hyperconic, if $\H$ satisfies Property (\texttt{*}).
Note that in the following lemma, we do not require $n$ to be prime.
\begin{lemma} \label{lem}Let $\H$ be a pseudo-hyperoval in $\PG(3n-1,q)$, $q=2^h$, $h>1$. Assume that
\begin{itemize}\setlength{\itemsep}{-1pt}
\item the spread induced by a subset $\T$ of $q^n+1$ elements of $\H$ is Desarguesian,
\item $\overline{\H}$ satisfies Property (P2) for some element $\overline{E_1}$ of $\overline{\T}$,
%\item Conjecture \ref{conj1} holds,
\end{itemize} then the following statements hold:

\begin{itemize}
\item[(i)] the elation Laguerre plane $L(\overline{\O})$ where $\overline{\O}=\overline{\H}\backslash \{\overline{E_1}\}$ is isomorphic to
the Laguerre plane $(\P', \L', \C')$ embedded in $\pi$, with natural incidence, given by
\begin{itemize}\setlength{\itemsep}{-1pt}
\item[$\P'$:] the lines of $\pi$ different from $\ell_\infty$,
\item[$\L'$:] the points of $\ell_\infty$,
\item[$\C'$:] the $q^{2n}$ point-pencils of $\pi$ not containing $\ell_\infty$ and $q^{3n}-q^{2n}$ dual ovals such that $\ell_\infty$ extends all of them to a dual hyperoval,
\end{itemize}
where $\pi$ is the Desarguesian projective plane $\PG(2,q^n)$ obtained from the Andr\'e/Bruck-Bose construction obtained from the spread $\overline{\S_1}$ and $\ell_\infty$ is the line of $\pi$ corresponding to $\overline{E}_1$.
\item[(ii)] a dual oval $\overline{A}$ of the set $\C'$ is a dual conic with $\ell_\infty$ as its nucleus line.
\item[(iii)] $L(\overline{\O})$ is Miquelian.
\end{itemize}
\end{lemma}

\begin{proof}
(i) Embed the space $\PG(3n-1,q)$, containing $\overline{\O}$, as a hyperplane $H_\infty$ in $\PG(3n,q)$. Recall that $L(\overline{\O})$ is the incidence structure $(\P,\L,\C)$, with natural incidence, embedded in $\PG(3n,q)$ as follows:
\begin{itemize}\setlength{\itemsep}{-1pt}
\item[$\P:$] the $2n$-spaces meeting $H_\infty$ in an element of $\overline{\O}$,
\item[$\L:$] the elements of $\overline{\O}$,
\item[$\C:$] the points of $\PG(3n,q)$ not contained in $H_\infty$ (the affine points).
\end{itemize}

Consider a $2n$-space $\Pi$ of $\PG(3n,q)$ intersecting $H_\infty$ in $\overline{E_1}$. The elements of $\overline{\O}$ intersect $\overline{E_1}$ in the Desarguesian spread $\overline{\S_1}$. It follows that the (projective) Andr\'e/Bruck-Bose construction in $\Pi$, using $\overline{\S_1}$, defines a Desarguesian projective plane $\pi\cong\PG(2,q^n)$. The elements of $\overline{\S_1}$ correspond to the points of a line $\ell_\infty$ of $\pi$.
By intersecting the elements of $L(\overline{\O})$ with $\Pi$, we find the representation $(\P',\L',\C')$ of the Laguerre plane $L(\overline{\O})$ in the Desarguesian plane $\pi$ as given in the statement. For this, we identify every circle of $\C$ with the $q^n+1$ elements of $\P$ it contains and consider their intersection with $\Pi$. Then, an affine point contained in $\Pi$ corresponds to a point-pencil of $\pi$ not containing $\ell_\infty$. An affine point not contained in $\Pi$ will also correspond to a set of $q^n+1$ lines of $\pi$, different from $\ell_\infty$. However, since such an affine point does not belong to $\Pi$, any three of these lines will have empty intersection, hence they form a dual oval. Moreover, these $q^n+1$ lines intersect the line $\ell_\infty$ all in a different point, therefore each dual oval extends uniquely to a dual hyperoval by adding the line $\ell_\infty$.
(ii) Consider the affine point $P$ of $\PG(3n,q)\backslash \Pi$ corresponding to $\overline{A}$. Consider three lines $\ell_1, \ell_2, \ell_3$ of $\overline{A}$. These correspond to three elements of $\overline{\H}$, say $\overline{E_2}, \overline{E_3}$ and $\overline{E_4}$. Now, since $\overline{\H}$ satisfies Property (P2), we find that the unique $\overline{\mathbf{S}_{n-1,2}}$, say $\B$, through the $4$ $(2n-1)$-spaces $\overline{E_1}$, $\overline{E_2}, \overline{E_3}$ and $\overline{E_4}$ contains $q+2$ elements of $\overline{\H}$. 

The element $\overline{E_1}$ is contained in $\B$, and the projection from $P$ of the $q^2+q$ $(2n-1)$-spaces of $\B$, different from $\overline{E_1}$, onto the space $\Pi$ (used in the Andr\'e/Bruck-Bose construction) corresponds to $q^{2}+q$ lines of the plane $\pi$. Every such projected line intersects $\ell_\infty$ in a point which corresponds to one of the $q+1$ elements of the unique regulus in $\overline{E_1}$ through $\overline{E_1}\cap \overline{E_2} $, $\overline{E_1}\cap \overline{E_3}$ and $\overline{E_1}\cap \overline{E_4}$. This implies that the set of $(2n-1)$-spaces $\B$ corresponds to the set of lines of an $\F_{q}$-subplane in the Desarguesian plane $\pi$, which contains $\ell_\infty, \ell_1, \ell_2, \ell_3$ and $q-2$ other lines of $\overline{A}$. Since this is true for every choice of three distinct lines $\ell_1,\ell_2,\ell_3$ of $\overline{A}$, by Theorem \ref{Tim}, $\overline{A}$ is a dual conic with $\ell_\infty$ as its nucleus line.

(iii) We consider the dual $(\P'', \L'', \C'')$ of the incidence structure $(\P',\L',\C')$ and use part (ii) which states that the dual ovals in $\C$ are dual conics. Also note that the dual of the Desarguesian plane $\pi$ is also Desarguesian. Let the point $N$ be the dual of the line $\ell_\infty$, then $(\P'',\L'',\C'')$ is given by
\begin{itemize}\setlength{\itemsep}{-1pt}
\item[$\P''$:] the points of $\PG(2,q^n)$ different from $N$,
\item[$\L''$:] the lines of $\PG(2,q^n)$ containing $N$,
\item[$\C''$:] the $q^{2n}$ lines of $\PG(2,q^n)$ not containing $N$ and the $q^{3n}-q^{2n}$ conics in $\PG(2,q^n)$ having $N$ as their nucleus.
\end{itemize}

This is just the standard plane model for a Miquelian Laguerre plane of even order $q^n$.
\end{proof}

\subsection{The proof of the main theorem}\label{conclusion}
We will first prove a lemma which gives a connection between Properties (P1) and (P2).

\begin{lemma}\label{P} Let $\H$ be a pseudo-hyperoval in $\PG(3n-1,q)$, $q=2^h$, $h>1$, such that there is a subset $\O$ of $q^n+1$ elements of $\H$ inducing a Desarguesian spread. If $\H$ does not satisfy Property (P1), then $\overline{\H}$ satisfies (P2) for every element of $\overline{\O}$.
\end{lemma}
\begin{proof}
If the hyperoval $\H$ does not satisfy Property (P1), then clearly, it does not satisfy Property (P1)(ii). So for every $4$ elements $E_i$, $i=1,\ldots,4$ of $\H$, the unique $\S_{n-1,2}$ through $E_i$, $i=1,\ldots,4$ contains $q+2$ elements of $\H$. This implies that the unique $\overline{\mathbf{S}_{n-1,2}}$ through $\overline{E_i}$, $i=1,\ldots,4$ contains $q+2$ elements of $\overline{\H}$, so $\overline{\H}$ satisfies Property (P2) for all elements of $\overline{\O}$.
\end{proof}
\begin{theorem} \label{hoofd2} If  $\O$ is a pseudo-oval in $\PG(3n-1,q)$, $q=2^h$, $h>1$, $n$ prime, such that the spread induced by every element of $\O$ is Desarguesian, then $\O$ is elementary.
\end{theorem}
\begin{proof}By Theorem \ref{Thas PHO}, we may consider the unique pseudo-hyperoval $\H$ extending $\O$. Clearly, $\H$ satisfies the conditions of Lemma \ref{P}. This implies that either $\H$ satisfies Property (P1), and then the statement follows from Theorem \ref{case1} (and the fact that a subset of an elementary set is elementary), or $\overline{\H}$ satisfies Property (P2) for every element of $\overline{\O}$.

By Lemma \ref{lem}, $L(\overline{\O})$ is Miquelian, and by Lemma \ref{lem2}, $\H$ is a pseudo-hyperconic with $E$ corresponding to the nucleus $N$ of a conic $\C$ (hence $\O$ is elementary). Note that only for $q=4$ this possibility can occur, since it is impossible that the set $\C\cup \{N\}\setminus \{P\}$, where $P$ is a point of $\C$ is again a conic, if $q>4$.
\end{proof}
As a corollary, we state a similar statement for pseudo-hyperovals.
\begin{corollary} Let $\H$ be a pseudo-hyperoval in $\PG(3n-1,q)$, $q=2^h$, $h>1$, $n$ prime, such that the spread induced by $q^n+1$ elements of $\H$ is Desarguesian, then $\H$ is elementary.
\end{corollary}
\begin{proof} The subset $\O$ of elements inducing a Desarguesian spread is an elementary pseudo-oval by Theorem \ref{hoofd2}, suppose $\O$ is the field reduced oval $A$. There is a unique element extending $\O$ to a pseudo-hyperoval, so $\H\setminus \O$ must be the element corresponding the unique point of $\PG(2,q^n)$ extending $A$ to a hyperoval.
\end{proof}

%\begin{theorem}\label{hoofd2}
% Consider a pseudo-hyperoval  $\H$ a in $\PG(3n-1,q)$, $q=p^h$, $p$ and $n$ prime, $q>2$, such that the spread induced by every element of $\H$ is Desarguesian, then $\H$ is elementary.
%\end{theorem}
%\begin{proof} If the hyperoval $\H$ contains $4$ elements such that for no value of $t<n$, the $\FR_{q^t}$-subplane through these $4$ elements is crowded, then the theorem follows from Theorem \ref{case1}. So we assume that every $4$ elements of $\H$ are contained for some $t<n$ in a crowded $\FR_{q^t}$-subplane. This implies that every element $E$ of $\H$ satifies property $(\texttt{*})$. Moreover, if $n$ is prime, Conjecture \ref{conj1} is true by Theorem \ref{Tim}, so the assumptions of Lemma \ref{lem} hold for $\H$ and every element $E$ of $\H$. By Lemma's \ref{miq} and \ref{lem2}, $\H$ is a pseudo-hyperconic with $E$ corresponding to the nucleus. Note that only for $q=4$ this possibility can occur, since it is impossible that the set $\C\cup \{N\}\setminus X$, where $X$ is a point of $\C$ is again a conic, if $q>4$.
%\end{proof}

\subsection{The consequence of the main theorem for Laguerre planes}
\begin{lemma}\label{derivation}
A point $P$ of an elation Laguerre plane $\Ll=L(\overline{\O})$, where $\overline{\O}$ is a dual pseudo-oval in $\PG(3n-1,q)$, admits a Desarguesian derivation if and only if the spread $\S$, induced by the line of $L(\overline{\O})$ through $P$ is Desarguesian.
\end{lemma}

\begin{proof} Let $P$ be a point of $\Ll$, then $P$ is a $2n$-space through an element $E$ of $\overline{\O}$. The derived affine plane of order $q^{n}$ at the point $P$ of $\Ll$ consists of points $\P'$ and lines $\L'$ obtained as follows:
\begin{itemize}
\item[$\P':$] $2n$-spaces in $\PG(3n,q)$, not in $H_\infty$, through an element of $\overline{\O}\setminus\{E\}$,
\item[$\L':$] points in $P$ not in $H_\infty$, together with the elements of $\overline{\O}\setminus E$.
\end{itemize}
Now this affine plane clearly extends to a projective plane of order $q^n$ by adding the $q^n+1$ elements of $\S$ as points and the space $E$ as line at infinity. This projective plane is the dual of the plane obtained from the (projective) Andr\'e/Bruck-Bose construction starting from $\S$ and hence, is Desarguesian if and only if $\S$ is Desarguesian.
\end{proof}

If $\Ll$ is a Laguerre plane of odd order, then the main theorem of Chen and Kaerlein \cite{Chen} states that the existence of one point admitting a Desarguesian derivation forces $\Ll$ to be Miquelian. The following theorem which is a consequence of our main theorem gives a (much) weaker result in the case of even order Laguerre planes.

\begin{theorem} Let $\Ll$ be a Laguerre plane of order $q^n$ with kernel $K$, $|K|\geq q^{3n}(q-1)$, $n$ prime, $q>2$ even. Suppose that for every line of $\Ll$, there exists a point on that line that admits a Desarguesian derivation, then $\Ll$ is ovoidal and $|K|=q^{3n}(q^n-1)$.
\end{theorem}
\begin{proof} From the hypothesis on the size of $K$ and Lemma \ref{St}, we find that $q^{3n}$ divides the order of $T$, hence, by \cite[Theorem 2]{Steinke} $\Ll$ is an elation Laguerre plane. By Theorem \ref{Steinke} $\Ll$ can be constructed from a dual pseudo-oval $\overline{\O}$ in $\PG(3n-1,q)$, $n$ prime. From Lemma \ref{derivation}, we obtain that for every element of $\overline{\O}$ the induced spread is Desarguesian. By Theorem \ref{hoofd2}, $\overline{\O}$ is elementary. By Theorem \ref{hoofd} this implies that $\Ll$ is ovoidal. Finally, this implies by Lemma \ref{St} that $|K|=q^{3n}(q^n-1)$.
\end{proof}

{\bf Acknowledgment.} The authors thank Tim Penttila for sharing Theorem \ref{Tim} with us, and John Sheekey for his help with the proof of Theorem \ref{moeilijk}.

\end{document}